\documentclass[a4paper]{amsart}

\usepackage{amsmath,amssymb,amsthm}
\usepackage[all]{xy}
\usepackage{tikz}
\usepackage[dvipdfmx, colorlinks=true, backref=page]{hyperref}

\newtheorem{theorem}{Theorem}[section]
\newtheorem{proposition}[theorem]{Proposition}
\newtheorem{lemma}[theorem]{Lemma}
\newtheorem{corollary}[theorem]{Corollary}

\theoremstyle{definition}
\newtheorem{definition}[theorem]{Definition}
\newtheorem{example}[theorem]{Example}

\newtheorem{problem}[theorem]{Problem}

\theoremstyle{remark}

\numberwithin{equation}{section}

\SelectTips{cm}{}

\newcommand{\R}{\mathbb{R}}
\newcommand{\C}{\mathbb{C}}

\title{Jacobi identity in polyhedral products}

\author{Daisuke Kishimoto}
\address{Department of Mathematics, Kyoto University, Kyoto 606-8502, Japan}
\email{kishi@math.kyoto-u.ac.jp}

\author{Takahiro Matsushita}
\address{Department of Mathematical Sciences, University of the Ryukyus, Okinawa 903-0213, Japan}
\email{mtst@sci.u-ryukyu.ac.jp}

\author{Ryusei Yoshise}
\address{Faculty of Mathematics, Kyushu University, Fukuoka 819-0395, Japan}
\email{yoshise.ryusei.597@s.kyushu-u.ac.jp}



\date{\today}

\begin{document}

  \maketitle

  \begin{abstract}
    We show that a relation among minimal non-faces of a fillable complex $K$ yields an identity of iterated (higher) Whitehead products in a polyhedral product over $K$. In particular, for the $(n-1)$-skeleton of a simplicial $n$-sphere, we always have such an identity, and for the $(n-1)$-skeleton of a $(n+1)$-simplex, the identity is the Jacobi identity of Whitehead products ($n=1$) and Hardie's identity of higher Whitehead products ($n\ge 2$).
  \end{abstract}

  %%%%% Section 1 %%%%%

  \section{Introduction}\label{Introduction}

  Let $K$ be a simplicial complex with vertex set $[m]=\{1,2,\ldots,m\}$, and let $(\underline{X},\underline{A})=\{(X_i,A_i)\}_{i=1}^m$ be a collection of pairs of spaces indexed by vertices of $K$. If all $(X_i,A_i)$ are a common pair $(X,A)$, then we abbreviate $(\underline{X},\underline{A})$ by $(X,A)$. The \emph{polyhedral product} of $(\underline{X},\underline{A})$ over $K$ is defined by
  \[
    Z_K(\underline{X},\underline{A})=\bigcup_{\sigma\in K}(\underline{X},\underline{A})^\sigma
  \]
  where $(\underline{X},\underline{A})^\sigma=Y_1\times\cdots\times Y_m$ such that $Y_i=X_i$ for $i\in\sigma$ and $Y_i=A_i$ for $i\not\in\sigma$. A polyhedral product as introduced by Bahri, Bendersky, Cohen and Gitler \cite{BBCG} as a generalization of the moment-angle complex $Z_K=Z_K(D^2,S^1)$ and the Davis-Januszkiewicz space $DJ_K=Z_K(\C P^\infty,*)$, which are fundamental objects in toric topology \cite{DJ}. It is connected to broad areas in mathematics and has been studied in many context. See a comprehensive survey \cite{BBC} for details. In this paper, we study relations among iterated (higher) Whitehead products of the inclusions $\Sigma X_i\to Z_K(\Sigma\underline{X},*)$, where $(\Sigma\underline{X},*)=\{(\Sigma X_i,*)\}_{i=1}^m$.

  We explain a motivation of our study. There is a natural action of a torus $T$ of rank $m$ on $Z_K$, which is of particular importance because quasitoric manifolds, a topological counterpart of toric varieties, are the quotient of $Z_K$ by a certain subtorus of $T$. Then the Borel construction $Z_K\times_TET$ and the inclusion $Z_K\to Z_K\times_TET$ are fundamental in toric topology. It is well known that $DJ_K\simeq Z_K\times_TET$ and the above inclusion is identified with the map
  \begin{equation}
    \label{w}
    Z_K\to DJ_K
  \end{equation}
  induced from the composition of the pinch map $(D^2,S^1)\to(S^2,*)$ and the bottom cell inclusion $(S^2,*)\to(\C P^\infty,*)$. If $K$ is the boundary of a $n$-simplex, then $Z_K=S^{2n-1}$ and the map \eqref{w} is the Whitehead product  of the bottom cell inclusions $S^2\to DJ_K$ for $n=2$ and the higher Whitehead product for $n\ge 3$ in the sense of Porter \cite{P}. Then Buchstaber and Panov posed the following problem \cite[Problem 8.2.5]{BP}.

  \begin{problem}
    \label{problem 1}
    When $Z_K$ is a wedge of spheres, is the map \eqref{w} a wedge of iterated (higher) Whitehead products of the bottom cell inclusions $S^2\to DJ_K$?
  \end{problem}

  Remark that there are several classes of simplicial complexes whose moment-angle complexes decompose into a wedge of spheres, for some of which polyhedral products also decompose \cite{GPTW,GT1,GT2,GW,IK1,IK2,IK3,IK4,IK6,IK7}. Here are results on Problem \ref{problem 1}. Grbi\'{c}, Panov, Theriault and Wu \cite{GPTW} obtained an affirmative solution for flag complexes, Grbi\'{c} and Theriault \cite{GT3} for MF-complexes, Abramyan and Panov \cite{AP} for the substitution of simplicial complexes, and Iriye and the first author \cite{IK4} for totally fillable complexes. As long as we are concerned with (higher) Whitehead products of the bottom inclusions $S^2\to DJ_K$, we only need to consider the map $Z_K=Z_K(D^2,S^1)\to Z_K(S^2,*)$ induced from a pinch map $(D^2,S^1)\to(S^2,*)$, which is generalized to a map
  \begin{equation}
    \label{w 2}
    \widetilde{w}\colon Z_K(C\underline{X},\underline{X})\to Z_K(\Sigma\underline{X},*)
  \end{equation}
  induced from the pinch maps $(CX_i,X_i)\to(\Sigma X_i,*)$, where $(C\underline{X},\underline{X})=\{(CX_i,X_i)\}_{i=1}^m$. The result of Iriye and the first author \cite{IK5} mentioned above is actually on the generalized map \eqref{w 2}.

  As for Whitehead products, the Jacobi identity is obviously fundamental. Then in this paper, we consider:

  \begin{problem}
    \label{problem 2}
    Is there a relation among iterated (higher) Whitehead products of the inclusions $\Sigma X_i\to Z_K(\Sigma\underline{X},*)$?
  \end{problem}

  %We will consider iterated (higher) Whitehead products which split off from the map \eqref{w 2} and will derive an identity among them from a property of a simplicial complex $K$ by using results in \cite{IK5}.
  The main theorem (Theorem \ref{main}) shows that a certain property of a \emph{fillable complex}, introduced in \cite{IK4}, yields an identity among iterated (higher) Whitehead products defined by minimal non-faces of the complex. So instead of the main theorem, we present its corollary here, which needs less notions to state, while it is restrictive compared to the main theorem.

  \begin{theorem}
    [Corollary \ref{main sphere later}]
    \label{main sphere}
    Let $K$ be the $(n-1)$-skeleton of a simplicial $n$-sphere $S$ with $n$-simplices $\sigma_1,\ldots,\sigma_r$, each of which is given a contraction ordering. If each $X_i$ is a suspension, then there is an identity of iterated (higher) Whitehead products
    \[
      w_{\sigma_r}=\epsilon_1w_{\sigma_1}+\cdots+\epsilon_{r-1}w_{\sigma_{r-1}},
    \]
    where $\epsilon_1,\ldots,\epsilon_{r-1}=\pm 1$ such that $\partial\sigma_r=\partial(\epsilon_1\sigma_1+\cdots+\epsilon_{r-1}\sigma_{r-1})$ as a simplicial chain.
  \end{theorem}

  Remarks on Theorem \ref{main sphere} are in order, where details will be given in Sections \ref{Fillable complex} and \ref{Result}. First, every $\sigma_i$ has at least one contraction ordering. Second, each $w_{\sigma_i}$ is a restriction of the map \eqref{w 2} and is an iterated (higher) Whitehead product of inclusions $X_i\to Z_K(\underline{X},*)$ determined by $\sigma_i$ and its contraction ordering. We will give example computations (Examples \ref{Jacobi Hardie}, \ref{cross-polytope} and \ref{RP^2}) to derive explicit identities among iterated (higher) Whitehead products by using our results, and will see that the classical Jacobi identity of Whitehead products and Hardie's Jacobi identity of higher Whitehead products \cite{H} are recovered from Theorem \ref{main sphere} for $S$ being the boundary of a simplex.

  %%%%% Acknowledgement %%%%%

  \subsection*{Acknowledgement}

  The authors were partially supported by JSPS KAKENHI Grant Numbers 17K05248 and 19K03473 (Kishimoto), and 19K14536 (Matsushita).

  %%%%% Section 2 %%%%%

  \section{Fillable complex}\label{Fillable complex}

  This section recalls a fillable complex introduced in \cite{IK4} and shows its properties that we are going to use. For the rest of this paper, let $K$ denote a simplicial complex with vertex set $[m]=\{1,2,\ldots,m\}$.
%  We review fillable complexes introduced in \cite{IK4}, which play an important role in this paper. This section shows its properties that we are going to use. For the rest of this paper, let $K$ denote a simplicial complex with vertex set $[m]=\{1,2,\ldots,m\}$.

  We set notation and terminology for simplicial complexes. Let $|K|$ denote the geometric realization of $K$. For a non-empty subset $I\subset[m]$, let
  \[
    K_I=\{\sigma\in K\mid\sigma\subset I\}
  \]
  which is called the \emph{full subcomplex} of $K$ over $I$. A subset of $M\subset [m]$ with $|M|\ge 2$ is called a \emph{minimal non-face} of $K$ if $M$ is not a simplex of $K$ and $M-v$ is a simplex of $K$ for each $v\in M$. So $M\subset [m]$ with $|M|\ge 2$ is a minimal non-face of $K$ if and only if $K_M$ is the boundary of a simplex with vertex set $M$. Note that if $M_1,\ldots,M_r$ are minimal non-faces of $K$, then $K\cup M_1\cup\cdots\cup M_r$ is a simplicial complex containing $K$ as a subcomplex. Let $\overline{K}$ denote a simplicial complex obtained by adding all minimal non-faces to $K$. For a finite set $I$, let $\Delta(I)$ denote a simplex with vertex set $I$.

  We define a fillable complex.

  \begin{definition}
    A simplicial complex $K$ is called \emph{fillable} if there are minimal non-faces $M_1,\ldots,M_r$ of $K$ such that $|K\cup M_1\cup\cdots\cup M_r|$ is contractible. The set of minimal non-faces $\{M_1,\ldots,M_r\}$ is called a \emph{filling} of $K$.
  \end{definition}

  We give examples of fillable complexes.

  \begin{example}
    \label{simplex}
    The boundary of a simplex is a fillable complex with a filling consisting of the only one minimal non-face which is the whole vertex set. Moreover, every skeleton of a simplex if a fillable complex which possibly possesses several fillings. More generally, it is proved in \cite{IK4} that the Alexander dual of a shellable complex is fillable, where each skeleton of a simplex is the Alexander dual of a shellable complex.
  \end{example}

  \begin{example}
    \label{sphere}
    The $(n-1)$-skeleton $K$ of a simplicial $n$-sphere $S$ with $n$-simplices $\sigma_1,\ldots,\sigma_r$ is a fillable complex with filling $\{\sigma_1,\ldots,\sigma_{r-1}\}$. Indeed, $|K\cup\sigma_1\cup\cdots\cup\sigma_{r-1}|$ is homotopy equivalent to $|S|$ minus the center of $\sigma_r$, which is contractible.
  \end{example}

  We refine \cite[Proposition 2.12]{IK5}. For a minimal non-face $M\subset[m]$, let
  \[
    g_M\colon |K_M|\to|K|
  \]
  denote the inclusion. Since $K_M=\partial\Delta(M)$ and $|\partial\Delta(M)|\cong S^{|M|-2}$, $g_M$ is a map from a sphere $S^{|M|-2}$ into $|K|$.

  \begin{lemma}
    \label{suspension K}
    If $K$ is a fillable complex with filling $\{M_1,\ldots,M_r\}$, then the map
    \[
      \Sigma g_{M_1}\vee\cdots\vee\Sigma g_{M_r}\colon S^{|M_1|-1}\vee\cdots\vee S^{|M_r|-1}\to|\Sigma K|
    \]
    is a homotopy equivalence.
  \end{lemma}

  \begin{proof}
    Since $|K\cup M_1\cup\cdots\cup M_r|$ is contractible, there is a homotopy equivalence
    \[
      |\Delta(M_1)|/|\partial\Delta(M_1)|\vee\cdots\vee|\Delta(M_r)|/|\partial\Delta(M_r)|=|K\cup M_1\cup\cdots\cup M_r|/|K|\to|\Sigma K|
    \]
    whose restriction to $|\Delta(M_i)|/|\partial\Delta(M_i)|=S^{|M_i|-1}$ is the suspension of $g_{M_i}$ for each $i=1,\ldots,r$. Then the proof is done.
  \end{proof}

  The following corollary is immediate from Lemma \ref{suspension K}, which will be used later without mentioning.

  \begin{corollary}
    If a fillable complex has two fillings $\{M_1,\ldots,M_r\}$ and $\{N_1,\ldots,N_s\}$, then $r=s$.
  \end{corollary}

  We will need to find identities among maps $\Sigma^2g_\sigma$, for which the following purity will be quite useful. We say that a filling $\{M_1,\ldots,M_r\}$ of a fillable complex $K$ is \emph{pure} if $|M_1|=\cdots=|M_r|$. Example \ref{sphere} implies that the $(n-1)$-skeleton of a simplicial $n$-sphere admits several pure fillings. We prove a homological condition that guarantees a relation among maps $g_M$.

  \begin{lemma}
    \label{pure filling}
    Let $K$ be a fillable complex with two pure fillings $\{M_1,\ldots,M_r\}$ and $\{N_1,\ldots,N_r\}$. Then there is an identity
    \[
      \partial N_i=a_1\partial M_1+\cdots+a_r\partial M_r
    \]
    in the simplicial chain complex of $\overline{K}$ for each $i=1,\ldots,r$. Moreover, we also have
    \[
      \Sigma^2 g_{N_i}=a_1\Sigma^2 g_{M_{1}}+\cdots+a_r\Sigma^2 g_{M_{r}}.
    \]
  \end{lemma}

  \begin{proof}
    The first claim is proved by Lemma \ref{suspension K}. By construction, the Hurewicz image of each $\Sigma g_{M_i}\in\pi_*(|\Sigma K|)$ is the suspension of the homology class of $K$ represented by a cycle $\partial M_i$. Of course, the same is true for $\Sigma g_{N_i}$. Then by the Hurewicz theorem, the second claim is proved.
  \end{proof}

  We recall from \cite{IK5} a contraction ordering for a minimal non-face. The following proposition is proved in \cite[Proposition 2.13]{IK5} by assuming $K$ is fillable. However, we can see that such an assumption is not necessary.

  \begin{proposition}
    \label{contraction ordering}
    For any minimal non-face $M$ of $K$, there are trees $T_1,\ldots,T_k$ satisfying the following conditions:
    \begin{enumerate}
      \item $\Delta(M)\cup T_1\cup\cdots\cup T_k$ is a subcomplex of $\overline{K}$ with vertex set $[m]$;

      \item $T_i\cap\Delta(M)$ is a vertex for each $i=1,\ldots,k$.

      \item $T_i \cap T_j = \emptyset$ if $i \ne j$.
    \end{enumerate}
  \end{proposition}

  \begin{proof}
    Since $\Delta(M)$ is connected, it includes a maximal tree $T$ whose vertex set is $\sigma$. It is easy to see that $T$ can be extended to a maximal tree $T'$ of $\overline{K}$ such that if we remove all edges of $T$ from $T'$, then we get a forest $T_1\sqcup\cdots\sqcup T_k$ where $T_i\cap\Delta(M)$ is a vertex for each $i=1,\ldots,k$. Since $\overline{K}$ is connected, the vertex set of $T'$ is $[m]$, so that $\Delta(M)\cup T_1\cup\cdots\cup T_k$ is a subcomplex of $\overline{K}$ with vertex set $[m]$. Thus the proof is done.
  \end{proof}

  Let $T$ be a tree with a distinguished root. We say that an edge of $T$ is free if one of its vertices is a leaf. By contracting free edges inductively, we can contract $T$ onto its root, and such a contraction is identified with an ordering of non-root vertices of $T$. We call such an ordering a \emph{contraction ordering}. Let $M,T_1,\ldots,T_k$ be as in Proposition \ref{contraction ordering}. Then by assuming that $T_i$ is given a root $T_i\cap\Delta(M)$, we get a contraction ordering of $T_i$. Joining such contraction orderings of $T_1,\ldots,T_k$ yields an ordering on $[m]-M$, which we call a contraction ordering of $M$. Then by Proposition \ref{contraction ordering}, we get:

  \begin{corollary}
    \label{contraction ordering 2}
    Every minimal non-face of a simplicial complex admits a contraction ordering.
  \end{corollary}

  %%%%% Section 3 %%%%%

  \section{Results}\label{Result}

  First, we recall from \cite{IK4} the fat-wedge filtration of a polyhedral product and necessary results on them. Let $\underline{X}=\{X_i\}_{i=1}^m$ be a collection of pointed spaces. For $i=0,1,\ldots,m$, let
  \[
    Z_K^i(C\underline{X},\underline{X})=\{(x_1,\ldots,x_m)\in Z_K(C\underline{X},\underline{X})\mid \text{at least }m-i\text{ of }x_k\text{ are basepoints}\}.
  \]
  Then we get a filtration
  \[
    *=Z_K^0(C\underline{X},\underline{X})\subset Z_K^1(C\underline{X},\underline{X})\subset\cdots\subset Z_K^{m-1}(C\underline{X},\underline{X})\subset Z_K^m(\underline{X})=Z_K(C\underline{X},\underline{X})
  \]
  which is the \emph{fat-wedge filtration} of $Z_K(\underline{X})$. It is proved in \cite{IK4} that
  \[
    Z_K^i(C\underline{X},\underline{X})/Z_K^{i-1}(C\underline{X},\underline{X})=\bigvee_{\substack{I\subset[m]\\|I|=i}}|\Sigma K_I|\wedge\widehat{X}^I
  \]
  where $\widehat{X}^I=X_{i_1}\wedge\cdots\wedge X_{i_k}$ for $I=\{i_1<\cdots<i_k\}$. In \cite{IK4}, several criteria for splitting the fat-wedge filtration are given in terms of $K$. In particular, we have the following decomposition. See \cite{BBC,IK2,IK3,IK6,IK7} for applications of the decomposition.

  \begin{theorem}
    \label{decomposition}
    If $K$ is a fillable complex, then there is a homotopy equivalence
    \[
      Z_K(C\underline{X},\underline{X})\simeq Z_K^{m-1}(C\underline{X},\underline{X})\vee(|\Sigma K|\wedge\widehat{X})
    \]
    where $\widehat{X}=X_1\wedge\cdots\wedge X_m$.
  \end{theorem}

  Next, we recall the result of \cite{IK5} which describes the map \eqref{w 2} in terms of iterated (higher) Whitehead products by using a homotopy decomposition in Theorem \ref{decomposition}. Let $e_i\colon\Sigma X_i\to Z_K(\Sigma\underline{X},*)$ denote the inclusion. For a minimal non-face $M\subset[m]$ of $K$, we can define the (higher) Whitehead product $w(M)$ of $e_i$ with $i\in M$. Namely, $w(M)$ is the composite
  \[
    \Sigma^{|M|-1}\widehat{X}^M\simeq Z_{\partial\Delta(M)}(C\underline{X}_M,\underline{X}_M)\xrightarrow{\widetilde{w}}Z_{\partial\Delta(M)}(\Sigma\underline{X}_M,*)\xrightarrow{\rm incl}Z_K(\Sigma\underline{X},*)
  \]
  where $\underline{X}_M=\{X_i\}_{i\in M}$. Moreover, given a contraction ordering $[m]-M=\{i_1<\ldots<i_k\}$, we can also define an iterated (higher) Whitehead product
  \[
    w_M=[\ldots[[w(M),e_{i_1}],e_{i_2}],\ldots,e_{i_k}]\circ\rho\colon\Sigma^{|M|-1}\widehat{X}\to Z_K(\Sigma\underline{X},*)
  \]
  where $\rho$ is the permutation $(1,2,\ldots,m)\mapsto(j_1,\ldots,j_l,i_1,\ldots,i_k)$ for $M=\{j_1<\cdots<j_l\}$. We recall state the main result of \cite{IK5}.

  \begin{theorem}
    \label{Whitehead product}
    Let $K$ be a fillable complex with filling $\{M_1,\ldots,M_r\}$ such that each of $M_i$ is equipped with a contraction ordering. Suppose that each $X_i$ is a suspension. Then the composite
    \begin{multline*}
      \Sigma^{|M_i|-1}\widehat{X}\xrightarrow{\Sigma g_{M_i}\wedge 1}|\Sigma K|\wedge\widehat{X}\xrightarrow{\rm incl}Z_K^{m-1}(C\underline{X},\underline{X})\vee(|\Sigma K|\wedge\widehat{X})\\
      \simeq Z_K(C\underline{X},\underline{X})\xrightarrow{\widetilde{w}}Z_K(\Sigma\underline{X},*)
    \end{multline*}
    is the iterated (higher) Whitehead product $w_{M_i}$.
  \end{theorem}

  Now we can state the main theorem of this paper.

  \begin{theorem}
    \label{main}
    Let $K$ be a fillable complex with two fillings $\{M_1,\ldots,M_r\}$ and $\{N_1,\ldots,N_r\}$ such that each of $M_i$ and $N_i$ is equipped with a contraction ordering. Suppose that each $X_i$ is a suspension. If
    \[
      \Sigma^2 g_{N_i}=a_1\Sigma^2 g_{M_1}+\cdots+a_r\Sigma^2 g_{M_r},
    \]
    then there is an identity among iterated (higher) Whitehead products
    \[
      w_{N_i}=a_1w_{M_{1}}+\cdots+a_rw_{M_{r}}.
    \]
  \end{theorem}

  \begin{proof}
    Since $\Sigma^2 g_{N_i}=a_1\Sigma^2 g_{M_1}+\cdots+a_r\Sigma^2 g_{M_r}$ and each $X_i$ is a suspension, we have
    \[
      \Sigma g_{N_i}\wedge 1_{\widehat{X}}=a_1\Sigma g_{M_{1}}\wedge 1_{\widehat{X}}+\cdots+a_r\Sigma g_{M_{r}}\wedge 1_{\widehat{X}}.
    \]
    Thus the identity in the statement is obtained by Theorem \ref{Whitehead product}.
  \end{proof}

  By Lemma \ref{pure filling} and Theorem \ref{main}, we get:

  \begin{corollary}
    \label{main pure}
    Let $K$ be a fillable complex with pure fillings $\{M_1,\ldots,M_r\}$ and $\{N_1,\ldots,N_r\}$ such that each of $M_i$ and $N_i$ is equipped with a contraction ordering. If each $X_i$ is a suspension., then for each $i=1,\ldots,r$,
    \[
      w_{N_i}=a_1w_{M_1}+\cdots+a_rw_{M_r},
    \]
    where $\partial N_i=\partial(a_1M_1+\cdots+a_rM_r)$ in the simplicial chain complex of $\overline{K}$.
  \end{corollary}

  \begin{proof}
    Clearly, there is an identity $\partial\sigma_r=\partial(\epsilon_1\sigma_1+\cdots+\epsilon_{r-1}\sigma_{r-1})$ as a simplicial chain for some $\epsilon_1,\ldots,\epsilon_r=\pm 1$. This relation readily implies $\Sigma^2 g(\sigma_r)=\epsilon_1\Sigma^2 g(\sigma_{1})+\cdots+\epsilon_{r-1}\Sigma^2 g(\sigma_{r-1})$. then the proof is finished by Theorem \ref{main}.
  \end{proof}

  By Example \ref{sphere}, Corollary \ref{main pure} specializes to:

  \begin{corollary}
    \label{main sphere later}
    Let $K$ be the $(n-1)$-skeleton of a simplicial $n$-sphere $S$ with $n$-simplices $\sigma_1,\ldots,\sigma_r$. If each $X_i$ is a suspension, then
    \[
      w_{\sigma_r}=\epsilon_1w_{\sigma_1}+\cdots+\epsilon_{r-1}w_{\sigma_{r-1}},
    \]
    where $\epsilon_1,\ldots,\epsilon_{r-1}=\pm 1$ such that $\partial\sigma_r=\partial(\epsilon_1\sigma_1+\cdots+\epsilon_{r-1}\sigma_{r-1})$ in the simplicial chain complex of $S$.
  \end{corollary}

  We give example computations of the above results.

  \begin{example}
    \label{Jacobi Hardie}
    Let $K$ be the $(m-3)$-skeleton of $\partial\Delta([m])$. Then $K$ has pure fillings
    \[
      \{[m]-i\mid i=1,2,\ldots,m-1\}\quad\text{and}\quad\{[m]-i\mid i=2,3,\ldots,m\}.
    \]
    Let $\sigma_i=[m]-i$ for $i=1,2,\ldots,m$. Then there is an identity
    \[
      \partial\sigma_m=\partial((-1)^{m+2}\sigma_1+\cdots+(-1)^{2m}\sigma_{m-1})
    \]
    in the simplicial chain complex of $S$. Therefore by Corollary \ref{main sphere later}, there is an identity
    \begin{equation}
      \label{Hardie}
      w_{\sigma_m}=(-1)^{m+2}w_{\sigma_1}+\cdots+(-1)^{2m}w_{\sigma_{m-1}}.
    \end{equation}
    Suppose that $m=3$ and $X_i=S^{p_i}$ for $i=1,2,3$. Then it is easy to see that the identity \eqref{Hardie} is exactly the same as the Jacobi identity of Whitehead products
    \[
      (-1)^{p_1p_3}[[e_1,e_2],e_3]+(-1)^{p_1p_2}[[e_2,e_3],e_1]+(-1)^{p_2p_3}[[e_3,e_1],e_2]=0.
    \]
    For $m>3$, if all $X_i$ are sphere, then \eqref{Hardie} coincides with Hardie's identity of higher Whitehead products \cite[Theorem 2.2]{H}. Then our identity is a combinatorial generalization of these two identities.
  \end{example}

  \begin{example}
    \label{cross-polytope}
    Recall that the $n$-dimensional cross-polytope is defined as the convex hull of $2n$ points
    \[
      (\pm 1,0,0,\ldots,0),\,(0,\pm 1,0,\ldots,0),\ldots,(0,0,\ldots,0,\pm 1)
    \]
    in $\R^n$. Then the $n$-dimensional cross polytope is the dual polytope of the $n$-dimensional hypercube. Let $S$ be the boundary of the $(n+2)$-dimensional cross-polytope. We may set the vertex set of $S$ to be $[2n+4]$ such that $\sigma \subset [2n+4]$ is a face of $S$ if and only if there is no $i = 1, 2,\cdots, n+2$ such that $\{ 2i-1, 2i \} \subset \sigma$. Thus a facet of $S$ is given by
    \[
      \sigma(\alpha_1, \cdots, \alpha_{n+2}) = \{ 2 - \alpha_1, 4- \alpha_2, \cdots, 2n+4 - \alpha_{n+2}\},
    \]
    where $\alpha_1, \cdots, \alpha_{n+2}$ are either $0$ or $1$. We consider identities among iterated (higher) Whitehead products in a polyhedral product over the $n$-skeleton $K$ of $S$. Then collections of facets of $S$
    \[
      \{ \sigma(\alpha_1, \cdots, \alpha_{n+2}) \; | \; (\alpha_1, \cdots, \alpha_{n+2}) \ne (1, \cdots, 1)\}
    \]
    and
    \[
      \{ \sigma(\alpha_1, \cdots, \alpha_{n+2}) \; | \; (\alpha_1, \cdots, \alpha_{n+2}) \ne (0, \cdots, 0)\}
    \]
    are pure fillings of $K$. For these pure fillings, we have
    \begin{align*}
      &\sum_{\alpha_1, \cdots, \alpha_{n+2}} (-1)^{\alpha_1 + \cdots + \alpha_{n+2}} \partial \sigma (\alpha_1, \cdots, \alpha_{n+2}) \\
      &=\sum_{\alpha_1, \cdots, \alpha_{n+2}} (-1)^{\alpha_1 + \cdots + \alpha_{n+2}} \bigg( \sum_{i=1}^{n+2} (-1)^{i-1} \{ 2-\alpha_1, \cdots \widehat{2i-\alpha_i}, \cdots, 2n+4 - \alpha_{n+2}\}\bigg) \\
      &=\sum_{\alpha_1, \cdots, \alpha_{n+2}} (-1)^{\alpha_1 + \cdots + \alpha_{n+2} + i - 1} \{ 2-\alpha_1, \cdots, \widehat{2i-\alpha_i}, \cdots, 2n+4-\alpha_{n+2}\}\\
      &=0
    \end{align*}
    in the simplicial chain complex of $S$. Then by Corollary \ref{main sphere later}, we get an identity
    \[
      w_{\sigma(1, \cdots, 1)} = \sum_{(\alpha_1, \cdots, \alpha_{n+2}) \ne (1,\cdots, 1)} (-1)^{\alpha_0 + \cdots + \alpha_n + n + 1} w_{\sigma(\alpha_1, \cdots, \alpha_{n+2})}.
    \]
  \end{example}

  \begin{example}
    \label{RP^2}
    This example considers a simplicial complex which is not a skeleton of a simplicial sphere. Let $K$ be the following graph, where we identify verticed and edges with the same names.
    
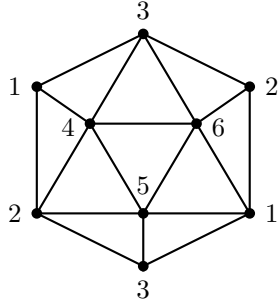
\begin{center}
      \begin{tikzpicture}[x=0.7cm, y=0.7cm, thick]
        \draw(0,0.1)--(-2,1.1)--(-2,3.5)--(0,4.5)--(2,3.5)--(2,1.1)--(0,0.1);
        \draw(-2,1.1)--(2,1.1)--(0,4.5)--(-2,1.1);
        \draw(-2,3.5)--(-1,2.8);
        \draw(2,3.5)--(1,2.8);
        \draw(-1,2.8)--(1,2.8)--(0,1.1)--(-1,2.8);
        \draw(0,0)--(0,1.1);
        \fill[black](-1,2.8)circle(2pt)node[below=1.5pt, left=2pt]{$4$};
        \fill[black](1,2.8)circle(2pt)node[below=1.5pt, right=2pt]{$6$};
        \fill[black](0,1.1)circle(2pt)node[above=3.5pt]{$5$};
        \fill[black](0,0.1)circle(2pt)node[below=2pt]{$3$};
        \fill[black](-2,1.1)circle(2pt)node[left=2pt]{$2$};
        \fill[black](2,1.1)circle(2pt)node[right=2pt]{$1$};
        \fill[black](-2,3.5)circle(2pt)node[left=2pt]{$1$};
        \fill[black](2,3.5)circle(2pt)node[right=2pt]{$2$};
        \fill[black](0,4.5)circle(2pt)node[above=2pt]{$3$};
      \end{tikzpicture}
    \end{center}
    Then $K$ is the 1-skeleton of a six vertex triangulation of $\R P^2$. We abbreviate $\{i,j,k\}$ by $ijk$ for $i,j,k=1,2,\ldots,6$. Let
    \[
      \mathcal{F}=\{124,\,126,\,134,\,135,\,156,\,235,\,236,\,245,\,346,\,456\}.
    \]
     Since $(\mathcal{F}-\sigma)\cup\{123\}$ is a filling of $K$ for each $\sigma\in\mathcal{F}$, $K$ is fillable. In the simplicial chain complex of $\overline{K}$, we have
    \begin{align*}
      \partial (456) &= 2 \partial (123) - \partial (124) - \partial (126) + \partial(134) + \partial (135)\\
      &\quad+ \partial (156) - \partial (235) - \partial (236) + \partial (245) - \partial (346).
    \end{align*}
    %This readily implies that there is an identity
    %\begin{align*}
    %  \Sigma^2g_{456}&=2\Sigma^2g_{123}+\Sigma^2g_{134}-\Sigma^2g_{345}-\Sigma^2g_{235}-\Sigma^2g_{125}\\
    %  &\quad-\Sigma^2g_{156}+\Sigma^2g_{136}-\Sigma^2g_{236}+\Sigma^2g_{246}-\Sigma^2g_{124}.
    %\end{align*}
    Then by Corollary \ref{main pure}, we obtain
     \[
      w_{456}=2 w_{123} - w_{124} - w_{126} + w_{134} + w_{135} + w_{156} - w_{235} - w_{236} + w_{245} - w_{346}.
    \]
    This shows that coefficients in the identity are not $\pm 1$ in general, while they are $\pm 1$ if $K$ is the $(n-1)$-skeleton of a simplicial $n$-sphere as in Corollary \ref{main sphere later}.
  \end{example}

\end{document}